\documentclass{amsart}
\usepackage{enumerate}
\usepackage{amsmath}
\usepackage{amsfonts}
\usepackage{amssymb}

\newtheorem{theorem}{Theorem}[section]
\newtheorem{lemma}[theorem]{Lemma}

\theoremstyle{definition}
\newtheorem{definition}[theorem]{Definition}
\newtheorem{example}[theorem]{Example}
\newtheorem{examples}[theorem]{Examples}

\newtheorem{corollary}[theorem]{Corollary}

\theoremstyle{remark}
\newtheorem{remark}[theorem]{Remark}

\theoremstyle{remarks}
\newtheorem{remarks}[theorem]{Remarks}
\newtheorem{proposition}[theorem]{Proposition}

\numberwithin{equation}{section}

\begin{document}
\title[  Products of Idempotents]{Products of Idempotents in Banach Algebras of Operators}

\author{S. K. Jain}
\address{S.K. Jain }  
\email{jain@ohio.edu \; Webpage address: sites.ohio.edu/jain} 

\author{Andre Leroy}
\address{Andr\'e Leroy}
\email{leroy55@gmail.com \;  Webpage address: http://leroy.perso.math.cnrs.fr/}

\author{Ajit Iqbal Singh}
\address{Ajit I. Singh INSA Emeritus Scientist, The Indian National Science Academy, Bahadur Shah Zafar Marg, New Delhi, 110 002, India}
\email{ajitis@gmail.com}
%
%

\begin{abstract}
	Let $X$ be a Banach space and $\mathcal A$ be the Banach algebra $B(X)$ of bounded (i.e. continuous) linear transformations (to be called operators) on 
	$X$ to itself.  Let $\mathcal E$ be the set of idempotents in $\mathcal A$ and $\mathcal S$ be the semigroup generated by $\mathcal E$ under composition as multiplication.  If $T\in \mathcal S$ with $0\ne T\ne I_{X}$ then $T$ has a local block representation of the form 
	$\begin{pmatrix}
		T_1 & T_2 \\
		0 & 0
	\end{pmatrix}$ on $X=Y\oplus Z$,  a topological sum of non-zero closed subspaces $Y$ and $Z$ of $X$, and any $A\in \mathcal A$ has the form $\begin{pmatrix}
		A_1 & A_2 \\
		A_3 & A_4
	\end{pmatrix}$
	with $T_1,A_1 \in \mathcal B(Y)$, $T_2,A_2\in B(Z,Y), A_3\in B(Y,Z)$, and $A_4 \in \mathcal B(Z)$. The purpose of this paper is to study conditions for $T$ to be in $\mathcal{S}$. 
\end{abstract}

\maketitle




%



\section{Introduction}\label{sec:1}

This paper is addressed to a general mathematician interested in Algebra, Linear Algebra and Operator theory.
The problem of expressing a non-invertible (i.e. singular) function $f$ on a non-empty set $X$ to itself as a (finite) product of idempotent functions on $X$ to itself was started by Howie \cite{Ho} in 1960's. He proved it  for finite sets $X$ and studied the situation for infinite sets in detail. The fact that taking $X$ as a basis of a linear space $\tilde{X}$ with $dim(\tilde{X})= \#X$ and function $f$ on $X$ to $X$ giving rise to linear transformation $T$ on $\tilde{X}$ to itself led Erdos \cite{Er} to study the problem for rings of $n\times n$ matrices over fields 
Howie's student Dawlings \cite{Da} studied the question for Banach algebras of bounded (i.e. continuous) linear transformations (to be called operators) on a Hilbert space $\mathcal{H}$ to itself.  He also analyzed Banach spaces $X$ with (or without) Schauder basis, considering the (topological) dimension to be smallest cardinality of a dense subset of $X$, and therefore, the cardinality of a Schauder basis, if it exists.  Several authors studied this further like Wu \cite{Wu} , Arias, Corach, and Maestripieri \cite{ACM}. Arias, Corach, and Maestripieri \cite{ACM} concentrated on products of two idempotents of $\mathcal A$. We find that the technique of local block representation gives several new results related to the expression of operators on certain classes as finite product of idempotents in $\mathcal A$ in infinitely many ways. At times, we get a better deal when $X$ is a Hilbert space $\mathcal{H}$. This constitutes the subject matter of this paper.

The paper is divided in four sections. Section 2 is concerned with basics and notation. Section 3 deals with properties of the set $\mathcal E$ of idempotents in $\mathcal A$ and the semigroup $\mathcal S$ generated by $\mathcal{E}$ with composition as multiplication. This includes the basic concept of local block representation for members of $\mathcal S$, other than $0$ and $I_X$, which coincides with the set $\mathcal{F}$ of finite products of members of $\mathcal{E}$ as well, other than $0$ and $I_X$. Section 4 develops the technique further to obtain more classes of operators expressible as products of two idempotents in infinitely many ways.  




%


\section{Preliminaries and motivation}\label{sec:2}

This  section collects basics from standard books, monographs or relevant papers such as Howie \cite{Ho}, Erdos \cite{Er}, Dawlings \cite{Da}, Lindenstrauss and Tzafriri \cite{LT}, survey article by Jain and Leroy (SURVEY) \cite{JL}, Alahmadi, Jain, Leroy, and Sathaye \cite{AJLS}, and Douglas \cite{Do} with a tinge of novelty.

\subsection{Preliminaries}\label{subsec:1} 

\begin{enumerate}
	\item [(i)] Let $X$ be a complex Banach space with  $\lVert . \rVert$. In case $X$ is a Hilbert space $\mathcal{H}$ with inner product $<.,.>$ we will usually write $\mathcal{H}$ in place of $X$. 
	
	\begin{enumerate}[\rm (a)]
		\item 
		Let $\mathcal{A}$ be the Banach algebra $\mathcal{B}(X)$ of bounded (i.e., continuous) linear transformations on $X$ to itself (to be called operators in this paper) with point wise addition and scalar multiplication and composition as multiplication and for $T\in\mathcal{A}$, $\lVert T\rVert=\sup \{\lVert Tx\rVert:\lVert x\rVert\le 1\}$.
		
		Let $I_X$ be the identity operator on $X$ to itself.
		
		\item 
		Let $\mathcal{E}$ be the set of idempotents in $\mathcal{A}$ and $\mathcal{S}$ the semigroup generated by $\mathcal{E}$ with composition as multiplication. Then $\mathcal{S}$ coincides with the set of (finite) products of members of $\mathcal{E}$.
		
		We  denote the set of such products other than $0$ and $I_X$ by $\mathcal{F}$.
		
		\item 
		For Banach spaces $X_1$ and $X_2$, $\mathcal{B}(X_1,X_2)$ will denote the Banach space of bounded (i.e., continuous) linear transformations (to be called operators) on $X_1$ to $X_2$. 
	\end{enumerate}
	
Then $\mathcal{B}(X_1, X_2)$ is a left $\mathcal{B}(X_1)$ Banach module and right $\mathcal{B}(X_2)$ Banach module. 

We shall use the same symbol $0$ for different $0$'s in respective spaces. 
	
	\item[(ii)] Let $\widetilde{\mathcal{K}}$ be the set of closed (linear) subspaces of $X$. 
	\begin{enumerate}
		\item [(a)] Suppose ${\rm dim} X \ge 2$. We write $X=Y\oplus Z$, a topological sum to mean that  $Y,Z \in \widetilde{\mathcal{K}}$, $\{0\}\ne Y \ne X$, $\{0\}\ne Z \ne X$, $Y\cap Z=\{0\}, Y+Z=X,$ the idempotent surjective linear maps  $p_Y:X\to Y$ and $p_Z= I_X-p_Y:X \to Z$ given by $p_Y(x)=y$, $p_Z(x)=z$ for $x=y+z$ with $x\in X$, $y\in Y$, $z\in Z$ are both continuous (i.e., bounded). In this case, $Y$ will be said to be topologically complemented (or in short complemented) in $X$ and $Z$ will be called a topological complement (in short, a complement) of $Y$ in $X$. 
The inclusion maps $i_Y:Y\to X$ and $i_Z:Z\to X$ are in $\mathcal{B}(Y,X)$ and $\mathcal{B}(Z,X)$  respectively and indeed isometric and satisfy $p_Y \circ i_Y=I_Y$ and $p_Z \circ i_Z=I_Z$.
		
		\item[(b)] Situation in (a) permits us to treat $T\in \mathcal{A}$ as a block matrix 
		$$T=\begin{bmatrix}
			T_1&T_2\\
			T_3&T_4
		\end{bmatrix}$$ in $X=Y\oplus Z$, a topological sum with $T_1=p_Y\circ T \circ i_Y \in \mathcal{B}(Y)=\mathcal{A}_1$, say, $T_2=p_Y\circ T \circ i_Z \in \mathcal{B}(Z,Y)=\mathcal{A}_2$, say,
		$T_3= p_Z \circ T \circ i_Y\in \mathcal{B}(Y,Z)=\mathcal{A}_3$, say and $T_4=p_Z\circ T \circ i_Z \in \mathcal{B}(Z)=\mathcal{A}_4$, say. We call it a block representation of $T$ in $X=Y\oplus Z$. We just mention that it is an example of a Morita context \cite{Mo}. 
		
		Let $\widehat{\mathcal{K}}$ be the subset of $\widetilde{\mathcal{K}}$ consisting of $Y$'s that possess complements in $X$.
		
		\item [(c)]	If $X=\mathcal{H}$, a Hilbert space then every closed subspace $\mathcal{K}$ of $\mathcal{H}$ is complemented and we may even have its orthogonal complement $\mathcal{L}=\mathcal{K}^{\perp}$. We write $\mathcal{H}=\mathcal{K} \oplus^{\perp} \mathcal{L}$ or simply $\mathcal{H}=\mathcal{K} \oplus \mathcal{L}$.
		
		\item [(d)] An old result of Lindenstrauss and Tzafriri \cite{LT} says that a Banach space $X$ in which linear closed subspace $Y$ is topologically complemented is isomorphic to a Hilbert space $\mathcal H$.  But in spite of that we may really not worry about such uncomplemented subspaces in $\tilde{\mathcal K}$ in our paper because of the special basic properties of $\mathcal{E}$ and $\mathcal{F}$.
		
		\item [(e)] For a closed subspace $Y$ of $X$ with $\{0\}\ne Y\ne X$, the (topological) dimension of the quotient Banach space $X/Y$ will be called the codimension of $Y$ in $X$. If $Y\in \widetilde{\mathcal{K}}$ is finite dimensional or has finite codimension, then  $Y\in \widehat{\mathcal{K}}$. This gives us the next interspersing proposition.
	\end{enumerate}	
\end{enumerate}
		\begin{proposition}\label{tm:2.2}
			Suppose ${\rm dim}X\ge 2$. If $Y$ is a closed subspace of $X$ with $\{0\}\ne Y \ne X$, then there exist $Y_0$ and $\widehat{Y}\in\widehat{\mathcal{K}}$ that satisfy $$\{0\}\ne Y_0\subset Y\subset \widehat{Y}\ne X.$$ Further, if $Y\not\in \widehat{\mathcal{K}}$, then there exist infinitely many choices for $Y_0$ as well as $\hat{Y}$; indeed,  there are  a strictly increasing sequence $(Y_n)_{n\in \mathbb{N}}$ in $\hat{\mathcal{K}}$ and a strictly decreasing sequence $(\hat{Y}_n)_{n\in\mathbb{N}}$ in $\hat{\mathcal{K}}$ that satisfy for $n\in\mathbb{N}$, $Y_n\subset Y\subset \hat{Y}_n$, $Y_n$ is complemented in $Y_{n+1}$ and $\hat{Y}_{n+1}$ is complemented in $\hat{Y}_n$.  
		\end{proposition}
		\begin{proof}
			It is enough to prove the last part. We have $Y$ is neither finite dimensional nor finite codimensional. We instantly have a desired sequence $(Y_n)_{n\in\mathbb{N}}$ of finite dimensional spaces contained in $Y$ which is strictly increasing. By Hahn Banach Theorem, there is a continuous linear functional $f\ne 0$ which is zero on $Y$. We let $\hat{Y}_1=\mathcal{N}(f)$, the nullspace of $f$. We can proceed by induction by taking $X=Y^n$ and obtain $Y^{n+1}$.
		\end{proof}

\begin{corollary}\label{cor:2.3}
	Let $X$ and $Y$ be as in Theorem \ref{tm:2.2} above. Then we have the following. 
	\begin{enumerate}[\rm (i)]
		\item 
		There exists a $0\ne A\in \mathcal{E}$, which is zero on $Y$.
		
		\item 
		There exist $0\ne T_0$ and $I_X\ne T_1\in \mathcal{E}$ with $\mathcal{R}(T_0)\subset Y\subset \mathcal{R}(T_1)$, where $\mathcal{R}$ stands for the range of an element of $\mathcal A$.
	\end{enumerate}
\end{corollary}
\begin{proof}
	\begin{enumerate}[\rm (i)]
		\item 
		We may take $A=I_X-p_{\widehat{Y}}$, coming from some complement of $\widehat{Y}$ in $X$.
		
		\item 
		We may take $T_0=p_{Y_0}$ and $T_1=p_{\hat{Y}}$ coming from some complements of $Y_0$ and $\hat{Y}$ in $X$.
	\end{enumerate}
\end{proof}

\subsection{Basic properties of $\mathcal{E}$ and $\mathcal{F}$ and examples} Consider $p_Y$ as in \ref{subsec:1} (ii) (a) above. Then, $p_Y$ is an idempotent in $\mathcal{A}$, $\mathcal{R}(p_Y)=Y$ and $\mathcal{N}(p_Y)=Z=\mathcal{R}(I_X-p_Y)$. Geometrically speaking, $p_Y$ is a projection on $Y$ along $Z$. They constitute the base for Cartesian co-ordinate system. As already noted in Section \ref{sec:1}, Erdos \cite{Er} showed that they are the building blocks for singular linear transformations of $\mathbb{C}^n$ to itself. But it is not so for infinite-dimensional Banach spaces or even Hilbert spaces. The following lemma and remarks thereafter explain it to begin with.

\begin{lemma}
	\begin{enumerate}[\rm (i)]
		\item 
		For $Q\in \mathcal{E}$ with $0\neq Q\neq I_X$. $\mathcal{N}(Q)$, the null space of $Q$ and $\mathcal{R}(Q)$, the
		range of $Q$, are both closed subspaces of $X$ that are neither $\{0\}$ nor $X$. Indeed, $X=\mathcal{R}(Q)\oplus \mathcal{N}(Q)$, a topological sum.
		
		\item 
		Let $F\in\mathcal{F}$, other than $0$ and $I_X$, say $F=Q_1\cdots Q_p$ with $p\in\mathbb{N}$, $Q_j\in\mathcal{E}$, $0\ne Q_j\ne I_X$ for $1\le j\le p$. Then $\mathcal{N}(Q_p)\subset \mathcal{N}(F)$ and $\{0\}\ne\mathcal{R}(F)\subset \mathcal{R}(Q_1)$. Also, $0\ne \mathcal{N}(F)\ne \{0\}$ and $\overline{\mathcal{R}(F)}$, the closure of $\mathcal{R}(F)$, is $\ne X$.
		
		\item Let $\mathcal{K}_0=\mathcal{R}(T)$ and $\mathcal{K}_1$ its closure in $X$. 
		\begin{enumerate}[\rm (a)]
			\item 
			 Let $0\ne T\in\mathcal{A}$. Then $T$ has a non-zero left annihilator if and only if $T$ has a non-zero idempotent left annihilator if and only if $\mathcal{K}_1\ne X$.  
			
			\item 
			Moreover, $T$ has a non-zero right annihilator if and only if $T$ has a non-zero idempotent right annihilator if and only if $\mathcal{N}(T)\ne \{0\}$.
			
			\item 
			As a consequence, any $F\in\mathcal{F}$ has non-zero idempotent left and right annihilators.
		\end{enumerate}
		
		\item 
		For $X=\mathbb{C}^n$, $n\ge 2$, $T\in \mathcal{A}$, $\mathcal{N}(T)\ne \{0\}$ if and only if $T$ is singular, i.e., not invertible in $\mathcal{A}$ if and only if $\mathcal{R}(T)\ne X$.  Furthermore, $\mathcal{R}(T)$ is closed. Also, the concepts of $T$ having a non-zero left annihilator coincides with that of $T$ having a non-zero right annihilator. 
	\end{enumerate}
\end{lemma}
\begin{proof}
	Part (i) above and (ii) are immediate from discussion in this section. For (iii), we appeal to Corollary \ref{cor:2.3}. To give an idea, we elaborate for (iii) (a). Suppose $T$ has a non-zero left annihilator $A$ in $\mathcal{A}$. Then, $AT=0$. So, $A$ is zero on $\mathcal{K}_0$. But $A$ is continuous, so, $A$ is zero on $\mathcal{K}_1$. Because $A\ne 0$, we have $\mathcal{K}_1\ne X$. On the other hand, suppose $\mathcal{K}_1\ne X$. We know $\mathcal{K}_1\ne X$. We know $\mathcal{K}_1\ne \{0\}$. So, by Corollary 2.2 (i), there exists $0\ne B\in\mathcal{E}$, which is zero on $\mathcal{K}_1$. So, $B$ is a non-zero idempotent left annihilator of $T$.	
	
	Finally (iv) comes from basic Linear Algebra.
\end{proof}

\begin{remark}
	The statements in Lemma 2.3 (iv) above are not true for the space $\mathcal{H}=l_2=l_2(\mathbb{N})=\{x=(x_n)_{n\in\mathcal{N}},\sum_{n\in\mathbb{N}}|x_n|^2<\infty\}$ with inner product $<x,y>=\sum_{n\in\mathbb{N}}x_n\bar{y_n}$  for $x,y\in l_2(\mathbb{N})$ as shown below. The discussion can be adapted to Banach spaces $X=l_p$, $1\leq p\leq \infty$ as well.
	
	\begin{enumerate}[\rm (i)]
		\item 
		Let $T\in\mathcal{A}$ be defined by \[(Tx)_j=\begin{cases} 
			0, & j=1 \\
			x_{j-1} & j\geq 2
		\end{cases}.
		\]
		\begin{enumerate}[\rm (a)]
			\item 
			Then $\mathcal{N}(T)=\{0\}$ and $\mathcal{R}(T)=\{x\in l_2:x_1=0\}$, which is closed but $\ne\mathcal{H}$. So, $T$ has a non-zero left annihilator, viz., the orthogonal projection $P_1$ given by $P_1(x)=x_1e_1$, $x\in l_2$ and its non-zero scalar multiples as well. But $T$ has no non-zero right annihilators.
			
			\item 
			As a consequence, $T$ is not a finite product of idempotents in $\mathcal{A}$, i.e., $T\not\in\mathcal{F}$.
		\end{enumerate}
			\item 
				Let $T\in\mathcal{A}$ be defined by $(Tx)_j=x_{j+1}$, for $j\in\mathbb{N}$ and $x\in\mathcal{H}$.
			
			\begin{enumerate}[\rm (a)]
				\item 
				Then $Te_1=0$. So, $\mathcal{N}(T)\ne \{0\}$ and $T$ has a non-zero right annihilator viz., $P_1$ as in (i) (a) above together with its non-zero scalar multiples. But $\mathcal{R}(T)=\mathcal{H}$. So, $T$ has no non-zero left annihilators.
				
				\item 
				As a consequence, $T$ is not a finite product of idempotents in $\mathcal{A}$, i.e., $T\not\in \mathcal{F}$.
			\end{enumerate}
			
			\item 
				Let $T\in\mathcal{A}$ be given by $(Tx)_j=\frac{1}{j}x_j$ for $j\in\mathbb{N}$ and $x\in l_2$.
			
			\begin{enumerate}[\rm (a)]
				\item 
				Then $\mathcal{N}(T)=\{0\}$ and therefore, $T$ has no non-zero right annihilators. Moreover, for $j\in\mathbb{N}$, $e_j\in\mathcal{R}(T)$ simply because $T(je_j)=e_j$. So, $\mathcal{R}(T)$ is dense in $\mathcal{H}$, i.e., $\overline{\mathcal{R}(T)}=\mathcal{H}$. So, $T$ has no non-zero left annihilators either.
				
				\item 
				Further, $\mathcal{R}(T)\ne \mathcal{H}$ because $x=(x_j)_{j\in\mathbb{N}}$ given by $x_j=\frac{1}{j}$ is in $\mathcal{H}$ but not in $\mathcal{R}(T)$. This gives another instance of $T\in\mathcal{A}$ with $\mathcal{N}(T)=\{0\}$, but $\mathcal{R}(T)\ne \mathcal{H}$ in contrast to the situation in $\mathbb{C}^n$, $n\in\mathbb{N}$.
				
				\item 
				$T$ is not a regular element of $\mathcal{A}$. Indeed, let, if possible, $G\in\mathcal{A}$ be such that $TGT=T$. Then $T(GT-I_\mathcal{H})=0$. But $\mathcal{N}(T)=\{0\}$. So, $GT=I_\mathcal{H}$. So, for $j\in\mathbb{N}$, $je_j=GT(je_j)$. So, $\lVert G(e_j)\rVert=j$ whereas $\lVert e_j\rVert=1$. So, $G$ is not bounded, a contradiction.
				
				Hence $\mathcal{A}$ is not von Neumann regular. This demands individual study of properties of $\mathcal{F}$ as we cannot use the literature on von Neumann regular rings.
				
				We do that in the next section.
			\end{enumerate}
		\end{enumerate}
\end{remark}

	The equivalence of statements in the last part of (iii) (a) and (b) may not hold for all closed subalgebras of $\mathcal{A}=\mathcal{B}(X)$. 
\begin{examples}
	 Let $\mathfrak{C}$ be the Banach algebra $C[0,1]$ of continuous scalar valued functions on the closed interval $[0,1]$ with the sup norm. Let $f,g\in\mathcal{C}$ be given by $f(t)=2t-1$ for $0\leq t\leq \frac{1}{2}$ and $f(t)=0$ for $\frac{1}{2}\leq t\leq 1$, whereas $g(t)=0$ for  $0\leq t\leq \frac{1}{2}$ and $2t-1$ for $\frac{1}{2}\leq t\leq 1$. Then $fg=0$. Now, let $X=\mathfrak{C}^2$. Then $\mathcal{A}=\mathcal{B}(X)$ can be identified with $M_2(\mathcal{B}(\mathfrak{C}))$. We consider the closed subalgebra $\mathcal{A}_1=M_2(\mathfrak{M}(\mathfrak{C}))$, where $\mathfrak{M}(\mathfrak{C})$ the algebra of endomorphisms of $\mathfrak{C}$ identified with $\mathfrak{C}$ itself. Then $T=\begin{bmatrix}
		1 & 0\\
		0 & f
	\end{bmatrix}$ has $A=\begin{bmatrix}
		0 & 0\\
		0 & g
	\end{bmatrix}$ as a non-zero left annihilator but $T$ has no non-zero idempotent left annihilator. To elaborate, the left annihilators of $T$ in $\mathcal{A}_1$ are given by $\begin{bmatrix}
		0 & b\\
		0 & d
	\end{bmatrix}$ with $bf=0=df$. But $\begin{bmatrix}
		0 & b\\
		0 & d
	\end{bmatrix}$ is an idempotent if and only if $bd=b$ and $d^2=d$. The only idempotents in $\mathfrak{C}$ are $0$ and $1$. Now, $d=0$ forces $b=0$ whereas $d=1$ does not satisfy $df=0$. So, $T$ has no non-zero idempotent left annihilators in $\mathcal{A}_1$.
\end{examples}

\begin{examples}	
	\begin{enumerate}[\rm (i)]
		\item 
		We consider the algebra $M_2(\mathbb{R})$ and express some nonzero singular matrices $T$ in $M_2(\mathbb{R})$ as finite products of idempotents on the right.
		
		\begin{enumerate}[\rm (a)]
			\item 
			For $0\ne a\in\mathbb{R}$, $$\begin{bmatrix}
				a & 0\\
				0 & 0
			\end{bmatrix}=\begin{bmatrix}
			1 & a-1\\
			0 & 0
			\end{bmatrix}\begin{bmatrix}
			1 & 0\\
			1 & 0
			\end{bmatrix}.$$
			
			\item 
			For $0\ne b,\ 0\ne c\in\mathbb{R}$,
			$$\begin{bmatrix}
				bc & 0\\
				0 & 0
			\end{bmatrix}=\begin{bmatrix}
				1 & b\\
				0 & 0
			\end{bmatrix}\begin{bmatrix}
				0 & 0\\
				0 & 1
			\end{bmatrix}\begin{bmatrix}
			1 & 0\\
			c & 0
			\end{bmatrix}.$$
			
			\item 
			For $0\ne a,\ b\in\mathbb{R}$
			$$\begin{bmatrix}
				a & ab\\
				0 & 0
			\end{bmatrix}=\begin{bmatrix}
				1 & a-1+b\\
				0 & 0
			\end{bmatrix}\begin{bmatrix}
				1-b & b(1-b)\\
				1 & b
			\end{bmatrix}.$$
		\end{enumerate}
		
		\item 
		For any ring $R$ with identity 1 and $n=2r$, $r\in\mathbb{N}$, $\mathbb{R}$ in (i) above can be replaced by $M_r(R)$ giving us factorizations in analogy with (i).
		
		\item 
		For any ring $R$ with identity 1, $n\in\mathbb{N}$, $n>2r,r$ in $\mathbb{N}$, any $a\in M_{r}(R)$ can be written as $bc$ with $b\in M_{r,n-r}(R)$ and $c\in M_{n-r,r}(R)$ such as $b=\begin{bmatrix}
			a & 0
		\end{bmatrix}$, $c=\begin{bmatrix}
		I_r\\0
		\end{bmatrix}$. So, we can have an analogue to (i)(b) by replacing 1 in the first row by $I_r$ and second row by $I_{n-r}$.
		
		\item 
		For a Hilbert space $\mathcal{H}=\mathcal{K}\oplus  \mathcal{L}$ with $\mathcal{K}\perp \mathcal{L}$, ${\rm dim}\mathcal{K}\leq {\rm dim}\mathcal{L}$, we can have a $\mathcal{J}:\mathcal{K}\rightarrow \mathcal{L}_1\subset \mathcal{L}$ an isometric isomorphism with $\mathcal{L}_1$ a closed subspace of $\mathcal{L}$. Then $\mathcal{J}^*\mathcal{J}=I_\mathcal{K}$. So, we may have the following analogue of (i).
		\begin{enumerate}[\rm (a)]
			\item 
			$$\begin{bmatrix}
				a & 0\\
				0 & 0
			\end{bmatrix}=\begin{bmatrix}
				I_\mathcal{K} & (a-I_\mathcal{K})\mathcal{J}^*\\
				0 & 0
			\end{bmatrix}\begin{bmatrix}
				I_\mathcal{K} & 0\\
				\mathcal{J} & 0
			\end{bmatrix},$$ for $a\in \mathcal{B}(\mathcal{K})$.
			
			\item 
			For $0\ne b\in\mathcal{B}(\mathcal{L,K})$, $0\ne c\in\mathcal{B}(\mathcal{K},\mathcal{L})$,
			$$\begin{bmatrix}
				bc & 0\\
				0 & 0
			\end{bmatrix}=\begin{bmatrix}
				I_\mathcal{K} & b\\
				0 & 0
			\end{bmatrix}\begin{bmatrix}
			0 & 0\\
			0 & I_\mathcal{L}
			\end{bmatrix}\begin{bmatrix}
				I_\mathcal{K} & 0\\
				c & 0
			\end{bmatrix}.$$
			
			\item 
			For $0\ne a\in\mathcal{B}(\mathcal{K}), b\in\mathcal{B}(\mathcal{L}, \mathcal{K})$, 
			$$\begin{bmatrix}
				a & ab\\
				0 & 0
			\end{bmatrix}=\begin{bmatrix}
				I_\mathcal{K} & (a-1)\mathcal{J}^*+b\\
				0 & 0
			\end{bmatrix}\begin{bmatrix}
				I_\mathcal{K}-b\mathcal{J} & b-b\mathcal{J}b\\
				\mathcal{J} & \mathcal{J}b
			\end{bmatrix}.$$
		\end{enumerate}
		
		\item 
		Indeed, in Corollary 3.11 in \cite{ACM}, the conditions given there on $T$ force $T$ to satisfy $\mathcal{R}(T)^\perp=\mathcal{N}(T)=\mathcal{R}(T^*)^\perp=\mathcal{N}(T)$ and $T=0$ on $\mathcal{R}(T)^\perp$ and therefore, have the form $\begin{bmatrix}
		T_1 & 0\\
		0 & 0
		\end{bmatrix}$ with $T_1$ invertible in $\mathcal{B}(\mathcal{K})$. Then (ii)(a) above renders an easy proof of a generalization of Corollary 3.11 \cite{ACM}.
	\end{enumerate}
\end{examples}

We give a useful factorization theorem, due to Douglas \cite[Theorem 1 and the last paragraph]{Do}.   It will be referred to as Theorem D.

\begin{theorem}
	\label{rem:4}
	Let $X_1, X_2, X_3$ be Banach spaces and $U\in\mathcal{B}(X_1,X_3)$, $V\in\mathcal{B}(X_2,X_3)$. Then $\mathcal{R}(U)\subset \mathcal{R}(V)$ if and only if $U=VW$ for some $W\in\mathcal{B}(X_1,X_2)$. In this case, there exists a unique $W_0\in\mathcal{B}(X_1,X_2)$ that satisfies $U=VW_0$ and $\mathcal{N}(U)=\mathcal{N}(W_0)$.
\end{theorem}		


\section{Properties of idempotents in $\mathcal{A}$ and their products}

\subsection{Basic concepts and set-up}\label{subsec:2} 

Let us first consider $0\ne T\in \mathcal{A}$ that has a nonzero left annihilator as well as a non-zero right annihilator. The subject of this paper is to see if we can decompose such an operator $T\in \mathcal{A}$ into (finite) product of idempotents. 
We start with the following rewording of a part of Lemma 2.3 for use here.
\begin{lemma}
	Let $0\ne T\in \mathcal{A}$ have a non-zero left annihilator and denote by $\mathcal{K}_0$ its range and $\mathcal{K}_1$, the closure of $\mathcal{K}_0$. Then there exists $\mathcal{K}$ in $\hat{\mathcal{K}}$ that satisfies $\{0\}\ne \mathcal{K}_0\subset \mathcal{K}_1\subset \mathcal{K} \subsetneq X$.  
\end{lemma}

This motivates us to work in the scenario set up in the previous section especially 2.1 and 2.2. We shall write $\mathcal{K}$ in place of $Y$ and $\mathcal{L}$ in place of $Z$ with the understanding that $\mathcal{L}$ will be an arbitrary but then fixed complement of $\mathcal{K}$ in $X$, and that in case $X =\mathcal{H}$, a Hilbert space, $\mathcal{L}=\mathcal{K}^\perp$ as indicated in 2.1 (ii) (c) above.

  Let us first make some observations.

\begin{lemma}
\label{classification of idempotent according to range}
Let $T=\begin{pmatrix}
	T_1 & T_2 \\
	T_3 & T_4
\end{pmatrix}$ be an operator in its block representation according to a decomposition $X=\mathcal{K}\oplus \mathcal{L}$. 
\begin{enumerate}[\rm (i)]
	\item T is an idempotent if and only if $T_1^2 + T_2T_3 = T_1, T_1T_2 +
	T_2T_4 = T_2, T_3T_1 + T_4T_3 = T_3$, and $T_3T_2 + T^2_4 = T_4$.
	\item 
	Let $T$ be an idempotent with $\mathcal{K}\subseteq \mathcal{R}(T)$.  Then 
		$T_1=I_\mathcal{K}$.
			$T_2T_3=0, T_2T_4=0, T_4T_3=0$, and
			$T_3T_2=T_4-T_4^2$.
	\item 
	\begin{enumerate}[\rm (a)]
		\item 
		$\mathcal{R}(T)\subseteq \mathcal{K}$ if and only if $T_3=0$ and $T_4=0$.
		
		\item 
		In case (a) happens, then:\\
\begin{itemize}
	\item $\mathcal{R}(T)=\mathcal{R}(T_1)+\mathcal{R}(T_2).$
	
	\item 
	$\mathcal{N}(T)= (\mathcal{N}(T_1)\oplus\mathcal{N}(T_2))+\{k+l:k\in\mathcal{K},l\in\mathcal{L},(k,l)=(0,0)\mbox{ or }T_1k=-T_2l\ne 0\}.$	
\end{itemize}		
	\end{enumerate}
	\item 
	 $T$ is an idempotent with $\mathcal{R}(T)\subset\mathcal{K}$ if and only if it is of the form 
	$\begin{pmatrix}
		T_1 & T_2 \\
		0 & 0
	\end{pmatrix}$ with 
	\begin{enumerate}[\rm (a)]
		\item 
		$T_1^2=T_1$ and
		
		\item 
		$T_1T_2=T_2$.
	\end{enumerate}
	  Under (a), the condition (b) can be replaced by any of the following:
	  \begin{enumerate}
	  	\item[\rm ($b^\prime$)] $T_2=T_1 B$ for some $B \in \mathcal{B}(\mathcal{L},\mathcal{K})$.
	  	
	  	\item[\rm ($b^{\prime\prime}$)] $\mathcal{R}(T_2)\subset \mathcal{R}(T_1)$.
	  	
	  	\item[\rm ($b^{\prime\prime\prime}$)]  $\mathcal{R}(T)=\mathcal{R}(T_1)$.
	  \end{enumerate}
	  
	\item  $T$ is an idempotent with $\mathcal{R}(T)= \mathcal{K}$ if and only if $T_3=0,T_4=0,$and $T_1=I_\mathcal{K}$, i.e., $T=\begin{bmatrix}
		I_\mathcal{K}&T_2\\
		0&0
	\end{bmatrix}$.
\end{enumerate}
\end{lemma}
\begin{proof}
(i) Direct computations give that $T$ is an idempotent if and only if 	 
	$T_1^2+T_2T_3=T_1, T_1T_2+T_2T_4=T_2,
	T_3T_1+T_4T_3=T_3, \mbox{and}\
	T_3T_2+T_4^2=T_4.$

 (ii) Since $K\subseteq \mathcal{R}(T)$ and $T^2=T$, we immediately obtain that $T_1=I_{K}$ and hence thanks to the equations above, the other equalities in (ii) follow. 

(iii) is clear.

(iv) Follows from (ii) and Theorem D.
\end{proof}
Let $\mathcal{E}_\mathcal{K}$ be the set of idempotents in $\mathcal{A}$ with range $\mathcal{K}$ and $\widetilde{\mathcal{E}_\mathcal{K}}$, the set of idempotents $Q$ in $\mathcal{A}$ with range $\mathcal{R}(Q)\subset \mathcal{K}$.

With this notation we easily get the following theorem.

\begin{theorem}
\label{tm:1}
\begin{enumerate}[\rm (i)]
		\item 
		$\mathcal{E}_\mathcal{K}$ is a non-commutative semigroup.
		
		\item 
		$\widetilde{\mathcal{E}_\mathcal{K}}$ is left as well as right $\mathcal{E}_\mathcal{K}$-module.
		
		
%
	\end{enumerate}
\end{theorem}
\begin{proof}
\begin{enumerate}
\item[(i)] Let $T=\begin{bmatrix}
			T_1 & T_2 \\
			T_3 & T_4
		\end{bmatrix}\in \mathcal{E}_\mathcal{K}$.  Since the range of $T$ is $\mathcal{K}$, the statement $(iv)$ in Lemma \ref{classification of idempotent according to range} shows that $T$ is of the form
		$T=\begin{bmatrix}
			I_\mathcal{K}&T_2\\
			0&0
		\end{bmatrix}$.  It is then obvious that these elements form a semigroup.
%
%
		\item[(ii)] The elements of $\widetilde{\mathcal{E}_\mathcal{K}}$ and $\mathcal{E}_\mathcal{K}$  are described in the statement (iii) and (iv) of the above lemma.  Let $T_1,T_1'$ idempotents in $\mathcal{B}(K)$ and $T_2,T_2'\in \mathcal{B}(L,K)$ with $T_1T_2=T_2$ and $T_1'T_2'=T_2'$.    We have
		
			\begin{equation*}
				\begin{bmatrix}
					T_1&T_1T_2\\
					0&0
				\end{bmatrix}\begin{bmatrix}
					I_\mathcal{K}&T_2'\\
					0&0
				\end{bmatrix}=\begin{bmatrix}
					T_1&T_1T_2'\\
					0&0
				\end{bmatrix},
			\end{equation*}
			
			\begin{equation*}
				\begin{bmatrix}
					I_\mathcal{K}&T_2'\\
					0&0
				\end{bmatrix}\begin{bmatrix}
					T_1&T_1T_2\\
					0&0
				\end{bmatrix}=\begin{bmatrix}
					T_1&T_1T_2\\
					0&0
				\end{bmatrix}.
			\end{equation*}		
This leads to the fact that $\widetilde{\mathcal{E}_\mathcal{K}}$ is left as well as right $\mathcal{E}_\mathcal{K}$-module
	\end{enumerate}
\end{proof}

We notice that $\{\widetilde{\mathcal{E}_\mathcal{K}}:\mathcal{K}\in\hat{\mathcal{K}}\}$ is an increasing family of sets of idempotents in $\mathcal{A}$ with union $\mathcal{E}$. This concept can be seen in literature and it has been used in different ways. Subsection 2.2 gives an idea.  We shall utilize it further for the problem at hand.


\begin{definition}\label{rem:2}
		Let $T\in\mathcal{F}$ with $T\not\in\mathcal{E}$. Then there is a smallest $t\ge 2$ such that $T=Q_1\cdots Q_t$ with each $Q_j\in\mathcal{E}$. We shall call $t$ the \textit{idempotent index} of $T$, in short $ii(T)$. For $T$ having idempotents index $t$ a factorization of $T$ into $t$ idempotents, say $E_1\cdots E_t$ will be said to be a \textit{minimal idempotent factorization}. 
\end{definition}

\begin{remarks}
	\begin{enumerate}[\rm (i)]
%
		\item 
		The idempotent index of $T\in\mathcal{E}$ will be taken to be 1 and for $T\in\mathcal{A}\setminus \mathcal{S}$, it will be taken to be $\infty$. 
		
		At times, we write $\mathcal{F}_t$ for $\mathcal{E}^t$ for $t\geq 2$.
		
		\item 
		For $t\in\mathbb{N}$, let $\mathcal{S}_t$ be the subset of $\mathcal{S}$ consisting of $T$ with idempotent index $t$. Then $\mathcal{S}_t$'s	are mutually disjoint and have union equal to $\mathcal{S}$.

		\item  For $s, t \in \mathbb{N}$, 
		$\mathcal{S}_{s+t}\subset \mathcal{S}_s \mathcal{S}_t$. 
		
		\item 
		Clearly $\mathcal{S}_t$ contained in symbol $\mathcal{E}^t$. Equality may not occur. 
		\begin{enumerate}[\rm (a)]
			\item 
			Examples 2.6 give a good idea to begin with. See Example 2.6 (i) (b) for $\mathcal{E}^3$ to be different from $\mathcal{S}_3$.
			
			\item Another instance is provided by non-zero mutually orthogonal projections in a Hilbert space $\mathcal{H}$.
		\end{enumerate}
		
		\item 
		Several other results in literature, for instance, \cite{JL} and \cite{AJLS} can be adapted to this general situation. We just give one more which helps in understanding idempotent index better, viz., part (vi) below and Theorem 3.6 below which is Theorem 3 (1) in \cite{JL} adapted to our set-up.
		

\item
Let $R$ be a ring with identity 1. Let $0\ne a \in R$ and $q\ne 1$ an idempotent in $R$ and $q^\prime =1-q$. Then $qa=0\Leftrightarrow a=q^\prime a$. So, for $b\in R$, $a=q^\prime b$ $\Leftrightarrow q^\prime (b-a)=0$, i.e., $c=b-a$ is a right annihilator of $q^\prime$. So, to begin with we can target elements like $b=a+c$ to be expressible as finite products of idempotents in $R$. Indeed, we also get that $ii(a)\leq 1+ii(c)$. 
\item 
Minimal representations for members of $\mathcal{S}\setminus \mathcal{E}$ are not unique. We have seen this before like in Example 2.6 (iv)(a) where we can have several maps $J$. Section 4 will give several Theorems and Examples for $t=2$. 
	\end{enumerate}
\end{remarks}

\begin{theorem}
	Let $T\in\mathcal{A}$ with block representation $T=\begin{bmatrix}
		T_1 & T_2\\
		0 & 0
	\end{bmatrix}$ in $X=\mathcal{K}\oplus\mathcal{L}$, a topological sum.
	\begin{enumerate}[\rm (i)]
		\item 
		If $T_1$ is a product of $s$ idempotents in $\mathcal{B}(\mathcal{K})$, then $T$ is a product of $s+1$ idempotents in $\mathcal{A}$. In this case, idempotent index of $T$ $\leq 1+$ idempotent index of $T_1$, in short $ii(T)\leq ii(T_1)+1$.
		
		\item 
		Inequality in (i) above can be strict.
	\end{enumerate}
\end{theorem}

\begin{proof}
	\begin{enumerate}[\rm (i)]
		\item 
		We follow the proof of Theorem 3 (1) in \cite{JL}. We have $T_1=E_1\cdots E_s$ with each $E_j$ an idempotent in $\mathcal{B}(\mathcal{K})$.
		
		Let for $1\leq j\leq s$, $Q_j=\begin{bmatrix}
			E_j & 0\\
			0 & I_\mathcal{L}
		\end{bmatrix}$ in $X=\mathcal{K}\oplus \mathcal{L}$ and let $Q_0=\begin{bmatrix}
		I_\mathcal{K} & T_2\\
		0 & 0
		\end{bmatrix}$. Then for $0\leq j\leq s$, $Q_j\in\mathcal{E}$ and $0\neq Q_j\ne I_X$. Now $T=Q_0Q_1\cdots Q_s$. So, $T\in \mathcal{E}^{s+1}$. By Remark 3.5 (iii), we have $ii(T)\leq ii(T_1)+1$.
		
		\item 
		We give some instances of the strict inequality.
		\begin{enumerate}[\rm (a)]
		\item 
			In Example 2.6 (i) (a), for $a\ne 1$, $ii(a)=\infty$ whereas $ii\begin{bmatrix}
				a & 0\\
				0 & 0
			\end{bmatrix}=2$.
		
		\item 
		We consider Example 2.6 (ii) and take $r\geq 3$. It is well-known (\cite{Er}, \cite{JL}, \cite{AJLS}) there is an $a\in M_r(\mathbb{R})$ with $ii(a)\geq 3$. By Example 2.6 (ii), $ii\begin{bmatrix}
			a & 0\\
			0 & 0
		\end{bmatrix}=2$.
		
		\item 
		We consider Example 2.6 (iv) and take $\mathcal{K}=l_2=\mathcal{L}$ and $\mathcal{H}=l_1\oplus l_2$. Next, we take $T_1$ to be any of the $T$'s in Remark 2.4 above. Then $ii(T_1)=\infty$. But by Example 2.6 (iv) (a), $ii\begin{bmatrix}
			T_1 & 0\\
			0 & 0
		\end{bmatrix}=2$.
		\end{enumerate}
	\end{enumerate}
\end{proof}

\begin{lemma}
	If $T\in \mathcal{F}\setminus\mathcal{E}$ is written as $T=Q_1\dots Q_t$ where $Q_1,\dots,Q_t$ are idempotents and $t$ is the idempotent index of $T$, then for any $1\leq j < t$ we must have 
	\begin{enumerate}[\rm (i)]
		\item 
		$Q_j$ and $Q_{j+1}$ cannot be in any semigroup $\mathcal{S}^\prime\subset\mathcal{E}$.
		
		\item 
		$K_{j+1}\setminus K_j\ne \emptyset$ and $K_j\setminus K_{j+1}\ne \emptyset$, where $K_s$ denotes $\mathcal{R}(Q_s)$, $1\leq s< t$.  
	\end{enumerate}
\end{lemma}
\begin{proof}
Indeed it follows from Theorem 3.3 (ii). But we also give an idea of a direct proof. It is enough to consider the case when $T=Q_1Q_2$, 
i.e. when $t=2$.  If we have $K_1\subset K_2$, then  for any $x\in X$ we have that $Q_2Q_1(x)=Q_1(x)$ and 
hence $Q_1Q_2Q_1Q_2(x)=Q_1Q_1Q_2(x)=Q_1(x)$, so that $Q_1Q_2$ is indeed an idempotent.  It is  then clear that $Q_1Q_2\in \mathcal{E}_{K_1}$, but this contradicts the fact that $t=2$ is the idempotent index of $T$.  One can show that if $K_2\subseteq K_1$, then $Q_1Q_2=Q_2$, this contradicts again the fact that $t=2$.        
\end{proof}

\begin{remark}
			The converse of the result in Lemma 3.7 (ii) is not true as can be seen by taking
		$Q_1=\begin{bmatrix}
			1 & a\\
			0 & 0
		\end{bmatrix}$ and $Q_2=
		\begin{bmatrix}
			0 & 0\\
			b & 1
		\end{bmatrix}$ with $ab=1$.
\end{remark}

%

\begin{remark}\label{rem:5}
	It will be helpful to list more idempotents in $\mathcal{A}$ in the form as in \ref{subsec:2} (i) simply by interchanging the roles of $\mathcal{K}$ and $\mathcal{L}$, $T_1$ and $T_4$, $T_2$ and $T_3$ in the statements in the discussion above.
	
	\begin{enumerate}[\rm (i)]
	\item 	Idempotents $Q$ in $\mathcal{A}$ with $\mathcal{R}(Q)\supset \mathcal{L}$, are given by $Q=\begin{bmatrix}
		T_1&T_2\\
		T_3&I_\mathcal{L}
	\end{bmatrix}$ that satisfy $T_2T_3=T_1(I_\mathcal{K}-T_1)=(I_\mathcal{K}-T_1)T_1, T_1T_2=0, T_3T_1=0,$ and $T_3T_2=0$.
	
	\item 
	$Q\in\mathcal{E}$ with $\mathcal{R}(Q)\subset \mathcal{L}$ are given by $Q=\begin{bmatrix}
		0&0\\
		DC&D
	\end{bmatrix}$ with $D\in\mathcal{B}(\mathcal{L})$ an idempotent, $C\in\mathcal{B}(\mathcal{K},\mathcal{L})$ i.e., $Q=\begin{bmatrix}
	0 & 0\\
	C_1 & D
	\end{bmatrix}$ with $D\in\mathcal{B}(\mathcal{L})$ an idempotent and $C_1\in\mathcal{B}(\mathcal{K},\mathcal{L})$ with $\mathcal{R}(C_1)\subset \mathcal{R}(D)$; further, $\mathcal{R}(Q)=\mathcal{R}(D)$.
	
	\item 
	Any $Q\in\mathcal{E}$ with $\mathcal{R}(Q)=\mathcal{L}$ is given by $Q=\begin{bmatrix}
		0&0\\
		C&I_\mathcal{L}
	\end{bmatrix}$ with $C\in\mathcal{B}(\mathcal{K},\mathcal{L})$. 
	\end{enumerate}
\end{remark}

We note that all these form a non-commutative semigroup with product $Q_1Q_2=Q_2$, viz., 
\begin{equation*}
	\begin{bmatrix}
		0&0\\
		C&I_\mathcal{L}
	\end{bmatrix}\begin{bmatrix}
	0&0\\
	C^\prime & I_\mathcal{L}
	\end{bmatrix}=\begin{bmatrix}
	0&0\\
	C^\prime & I_\mathcal{L}
	\end{bmatrix}
\end{equation*}
in line with Theorem \ref{tm:1} (i).

\subsection{Concepts of local block representation, dimension, and ``essentially contained in"}

We recast these concepts in the way we like for further use.

\begin{definition}
	\label{Decomposition for a non regular operator}
	Let $0\ne T\in \mathcal{A}$. 
	Assume that $T$ has a non-zero left annihilator. Let $\mathcal{K}_0=\mathcal{R}(T)$ and $\mathcal{K}_1$ it closure in $X$. Lemma 3.1 and Lemma 3.2 above give us a $\mathcal{K}\in\hat{\mathcal{K}}$ with $\mathcal{K}_1\subseteq\mathcal{K}$ and enable us to express $X=\mathcal{K}\oplus\mathcal{L}$, a topological sum  for any complement $\mathcal{L}$ of $\mathcal{K}$ in $X$ and have the block, representation of $T$ with respect to this decomposition as $T=\begin{bmatrix}
		T_1&T_2\\
		0&0
	\end{bmatrix}$ in $X=\mathcal{K}\oplus \mathcal{L}$. We call it the \textit{\bf Local Block representation} of $T$ with respect to the decomposition $\mathcal{K}\oplus \mathcal{L}$ of $X$, or simply a local block representation if no confusion can arise.  
	
	Notice that such an operator $T$ has a non-zero right annihilator if and only if at least one of the following conditions is satisfied
	\begin{enumerate}[\rm (a)]
		\item $\mathcal{N}(T_1)\ne\{0\}$,
		
		\item $\mathcal{N}(T_2)\ne\{0\}$, 
		
		\item $\mathcal{R}(T_1) \cap \mathcal{R}(T_2)\ne\{0\}$.
	\end{enumerate}
	We assume that it is so. 
\end{definition}

If $\mathcal{K}_1\in\mathcal{K}$, then we take $\mathcal{K}=\mathcal{K}_1$. If $\mathcal{K}_1\not\in\hat{\mathcal{K}}$, then by Proposition 2.1, there are infinitely many $\mathcal{K}$'s in $\hat{\mathcal{K}}$ available. We can carry out our discussion in any of them. There will be some interrelations which we do not go into at this moment.

 \begin{remark}
\begin{enumerate}[\rm (i)]
	\item 
	Let $G$ be a Banach space, $G_1$ a dense linear subspace of $G$, $S_1$ a non-empty subset of $G_1$, $\Gamma(S_1)$ the linear span of $S_1$ and $\Gamma_q(S_1)$ finite linear combination of members of $S_1$ using rational scalars only. We note that if $S_1$ is infinite, then $\#\Gamma_q(S_1)=\#S_1$. Also, the closures of $\Gamma(S_1)$ and $\Gamma_q(S_1)$ in $G_1$ coincide and so do the closure of $\Gamma(S_1)$ and $\Gamma_q(S_1)$ in $G$, let us denote them by $\overline{\Gamma(S_1)}^{G_1}$, $\overline{\Gamma_q(S_1)}^{G_1}$, $\overline{\Gamma(S_1)}$, and $\overline{\Gamma_q(S_1)}$, respectively. Finally, $\overline{\Gamma(S_1)}=G_1\Leftrightarrow\overline{\Gamma(S_1)}=G$. So, the smallest cardinality of $S_1$ such that $\overline{\Gamma(S_1)}=G$, equivalently, $\overline{\Gamma(S_1)}^{G_1}=G_1$ coincide. We call them the \textbf{dimension of $G$} as well as \textbf{dimension of $G_1$}, and denote them by ${\rm dim}G$ and ${\rm dim}(G_1)$, respectively.
	
	\item 
	If $G$ is a Hilbert space. Then ${\rm dim}(G)$ is the cardinality of any complete orthonormal basis of $G$. Moreover, if for Hilbert spaces $G$ and $G^\prime$ we have that if ${\rm dim}G\leq {\rm dim}G^\prime$, then there exists an isometric linear map $J$ on $G$ onto a closed subspace $G^{\prime\prime}\subset G^\prime$ and the adjoint $J^*:G^\prime \rightarrow G$ satisfies $J^*J=I_G$. Indeed, $J^*=J^{-1}\circ P_{G^{\prime\prime}}$, where $P_{G^{\prime\prime}}$ is the (orthogonal) projection on $G^\prime$ onto $G^{\prime\prime}$. 
	
	\item 
	To have a map like $J$ in the context of Banach spaces, we can impose the condition that $G^{\prime\prime}\in\hat{\mathcal{K}}$. To elaborate, we say that $G$ is \textbf{essentially contained in} $G^\prime$, in short $G\subset_e G^\prime$, if there exists a complemented closed subspace $G^{\prime\prime}$ of $G^\prime$ and a bicontinuous linear bijection $J: G\rightarrow G^{\prime\prime}$ and for notational convenience, denoted $J^{-1}\circ p_{G^{\prime\prime}}$ by $J^*$, where $p_{G^{\prime\prime}}$ is an idempotent in $\mathcal{B}(G^\prime)$ with range $G^{\prime\prime}$ coming from some complement of $G^{\prime\prime}$ in $G^\prime$. Then $J^*J=I_G$. We note that in this case, ${\rm dim}G={\rm dim}G^{\prime\prime}\leq {\rm dim}G^\prime$. The converse is true for Hilbert spaces.
	
	\item 
	We next note that if $A\in\mathcal{B}(G,G^\prime)$, then ${\rm dim}\mathcal{R}(A)\leq {\rm dim}G$. 
\end{enumerate}
\end{remark}

\begin{theorem}\label{rem:17}
	We shall freely use Remarks and results above. Let $t\in\mathbb{N}$ with $t\geq 2$.	Let $T \in\mathcal{S}_t$, for $1\leq j\leq t$, $E_j$ in $\mathcal{E}$ with $T=E_1...E_t$ a minimal decomposition of $T$. Let $\mathcal{K}=\mathcal{R}(E_1)$, $\mathcal{L}$, a complement of $\mathcal{K}$, for instance $\mathcal{L}=\mathcal{N}(E_t)$, in $X$. Then we have the following.
			\begin{enumerate}[\rm (i)]
				\item 
				 $E_1=\begin{bmatrix}
					I_\mathcal{K} & B\\
					0 & 0
				\end{bmatrix}$ for some $B\in\mathcal{B}(\mathcal{L},\mathcal{K})$, and $T=\begin{bmatrix}
				T_1 & T_2\\
				0 & 0
				\end{bmatrix}$ in $X=\mathcal{K}\oplus\mathcal{L}$, a topological sum for some $T_1\in\mathcal{B}(\mathcal{K})$, $T_2\in\mathcal{B}(\mathcal{L},\mathcal{K})$, its local block representation with respect to $X=\mathcal{K}\oplus\mathcal{L}$.
				
				\item 
				$S=E_2\cdots E_t=\begin{bmatrix}
				T_1-BC & T_2-BD\\
				C & D
				\end{bmatrix}=T_{B,C,D}$, say, for some $C\in\mathcal{B}(\mathcal{K},\mathcal{L})$ and $D\in\mathcal{B}(\mathcal{L})$.
				
				\item 
				$ii(\mathcal{S})=t-1$.
				
				\item 
				 ${\rm dim}\mathcal{R}(S)\geq {\rm dim}\mathcal{K}_1$, where $\mathcal{K}_1$ is the closure of $\mathcal{R}(T)$ in $X$.
			\end{enumerate}	
\end{theorem}			
			
%
%
		
\begin{proof}
(i) it is immediate from Lemma 3.2 (v) and Definition 3.10. 

(ii) Let $S= \begin{bmatrix}
	S_1 & S_2\\
	S_3 & S_4
\end{bmatrix}$ in $X=\mathcal{K}\oplus\mathcal{L}$. We use (i) and try to solve the equation $T=E_1S$. Direct computations give the solutions as stated in (ii). An alternative argument is given in Remark 3.13 (i) below.

(iii) To see this, the definition of $S$ gives that it has idempotent index $\le t-1$, if its idempotent index is $< t-1$ then that makes idempotent index of $T\le t-1$, a contradiction.

(iv) Because $T=E_1S$, we have that ${\rm dim}\mathcal{R}(T)\le {\rm dim}\mathcal{R}(S)$. But ${\rm dim}\mathcal{R}(T)= {\rm dim}\mathcal{K}_1$. So, the result follows.		
\end{proof}

\begin{remark}
	Let $0\ne T\in\mathcal{A}$ have a non-zero left annihilator and $T=\begin{bmatrix}
		T_1 & T_2\\
		0 & 0
	\end{bmatrix}$, a local block representation for $T$ in $X=\mathcal{K}\oplus \mathcal{L}$ with $\mathcal{K}\supset \mathcal{K}_1= \overline{\mathcal{K}_0}$, where $\mathcal{K}_0=\mathcal{R}(T)$. We intend to determine if $T$ is in $\mathcal{S}$ or not, if yes then we would like to obtain $ii(T)$ and factorizations  of $T$ into idempotents in $\mathcal{A}$. We display our line of thought and approach. We demonstrate it concretely for $\mathcal{F}_2$ in the next section.
	\begin{enumerate}[\rm (i)]
		\item 
		For any $B\in\mathcal{B}(\mathcal{L},\mathcal{K})$, $Q_B=\begin{bmatrix}
			I_\mathcal{K} & B\\
			0 & 0
		\end{bmatrix}$ satisfies $Q_BT=T$ and $Q_B\in \mathcal{E}_\mathcal{K}$. Right annihilators of $Q_B$ are given by $\begin{bmatrix}
			0 & 0\\
			C & D
		\end{bmatrix}$ in $X=\mathcal{K}\oplus \mathcal{L}$ with $C\in\mathcal{B}(\mathcal{K},\mathcal{L})$ and $D\in\mathcal{B}(\mathcal{L})$. So, in view of Remark 3.5 (vi) $Q_BT_{B,C,D}=T$, where $T_{B,C,D}=\begin{bmatrix}
			T_1-BC & T_2-BD\\
			C & D
		\end{bmatrix}$. This can be obtained by direct computation as well. So, we target $T_{B,C,D}$ for being in $\mathcal{S}$ or not, if yes, then determine its idempotent index $t_{B,C,D}$ and conclude that $T\in \mathcal{S}$ with $ii(T)\leq t_{B,C,D}+1$. 
			
			\item 
			$C=0,D=0$ gives $S=\begin{bmatrix}
				T_1&T_2\\
				0&0
			\end{bmatrix}=T$, which we do not want.
			 
			We would like to figure out $B\in\mathcal{B}(\mathcal{L},\mathcal{K})$ and $(C,D)\ne (0,0)$ in $\mathcal{B}(\mathcal{K},\mathcal{L})\times\mathcal{B}(\mathcal{L})$, for which $T_{B,C,D}\in\mathcal{F}$. As already noted in Lemma 2.3 (ii), this necessitates $\overline{\mathcal{R}(T_{B,C,D})}\ne X$, and $\mathcal{N}(T_{B,C,D})\ne \{0\}$ and $\mathcal{N}(T)\ne \{0\}$.
			
			\item 
			Indeed, if $\mathcal{K}_1^\prime =\overline{\mathcal{R}(T_{B,c,D})}\neq \mathcal{X}$, then we may repeat the discussion for $T$ for this $T_{B,C,D}$ and have an idempotent $Q^\prime$ with range $\mathcal{K}^\prime$ containing $\mathcal{K}_1^\prime$ as a factor of $T_{B,C,D}$. By Lemma 3.7, we would like $\mathcal{K}^\prime \setminus \mathcal{K}\ne \emptyset \ne \mathcal{K}\setminus \mathcal{K}^\prime$. This leads to an iteration process, which may end or just continue. 
			
			$\mathcal{K}^\prime\setminus \mathcal{K}\ne\emptyset$ can be ensured by $(C,D)\ne (0,0)$. For $\mathcal{K} \setminus \mathcal{K}^\prime \ne \emptyset$ we must have $(\mathcal{R}(T_1-BC)+\mathcal{R}(T_2-BD))\ne \mathcal{K}$.
			
			\item 
			Now,  
			\begin{align*}
				\mathcal{N}(T_{B,C,D})\ne 0 \Leftrightarrow & \;{\rm for\; some} \; k\in \mathcal{K} \, {\rm and}\, l \in \mathcal{L},
				0\ne  k+l\in\mathcal{N}(T_{B,C,D})\\
				\Leftrightarrow & (T_1-BC)k+(T_2-BD)l=0\mbox{ and 
				} Ck+Dl=0\\
				\Leftrightarrow & T_1k+T_2l-B(Ck+Dl)=0\mbox{ and 
				} Ck+Dl=0, 
				\\
				\Leftrightarrow & 
				T_1k+T_2l=0\mbox{ and } 
				Ck+Dl=0, {\rm for\; some\;} (k,l) \ne (0,0) \in \mathcal{K} \times 
				\mathcal{L}
			\end{align*}Anyone of the  following  conditions is sufficient, viz.,
			
			\noindent $\mathcal{N}(T_1)\cap \mathcal{N}(C)\ne \{0\}
			\mbox{or }  \mathcal{N}(T_2)\cap \mathcal{N}(D)\ne \{0\}
			\mbox{ or }  \exists\ k,l \mbox{ with } 0\ne T_1k=-T_2l,  0\ne Ck = -Dl$.
			
			This implies
			$\mathcal{N}(T_1)\ne\{0\}\;
			\mbox{or }   \mathcal{N}(T_2)\ne \{0\}
			\mbox{or }   \mathcal{R}(T_1)\cap \mathcal{R}(T_2)\ne \{0\}$ together with $ \mathcal{R}(C) \cap \mathcal{R}(D) \ne \{0\}$
		\end{enumerate}
	
\end{remark}


\section{New Classes of Elements of $\mathcal{F}_2$, the products of two idempotents in $\mathcal{A}$}\label{sec:3}

We shall give some general results and some special examples of operators in $\mathcal{A}$ that are products of two idempotents. We begin with some direct applications of discussion in Sections 2 and 3 above. We freely use the notation, terminology, and results therein. For this section, $0\ne T \ne I_X\in\mathcal{A}$ has non-zero left and right annihilators and $T= \begin{bmatrix}
	T_1 & T_2\\
	0 & 0
\end{bmatrix}$ in $X=\mathcal{K}\oplus\mathcal{L}$, is a local block representation for $T$, unless otherwise stated. We will simply write $T=\begin{bmatrix}
	T_1 & T_2\\
	0 & 0
\end{bmatrix}$, if no confusion can arise. Because $T_1= I_\mathcal{K}$ forces $T$ to be an idempotent with range equal to $\mathcal{K}$, we usually avoid repetition and take $T_1\ne I_\mathcal{K}$.

\subsection{Products of two idempotents by inspection.}

\begin{lemma}\label{tm:2}
	Let  $B\in\mathcal{B}(\mathcal{L},\mathcal{K})$, and $C\in\mathcal{B}(\mathcal{K},\mathcal{L})$. Then
		\item 
		$$\begin{bmatrix}
			BC&B\\
			0&0
		\end{bmatrix}\ \text{and}\ \begin{bmatrix}
		0 & 0\\
		C & CB
		\end{bmatrix}$$ are products of two idempotents in $\mathcal{A}$.
\end{lemma}

\begin{proof}
		Using Lemma 3.2 (v) and Remark \ref{rem:5} (iii), $$E_1=\begin{bmatrix}
			I_\mathcal{K} & B\\
			0 & 0
		\end{bmatrix}\ \text{and}\ E_2=\begin{bmatrix}
		0&0\\
		C&I_\mathcal{L}
		\end{bmatrix}$$ are idempotents in $\mathcal{A}$. Also, $$\begin{bmatrix}
		BC&B\\
		0&0
		\end{bmatrix}=E_1E_2\quad {\rm and}\quad \begin{bmatrix}
		0 & 0\\
		C & CB
		\end{bmatrix}=E_2E_1$$.
\end{proof}

\begin{theorem}\label{cor:1}
	Let $T=\begin{bmatrix}
		T_1&T_2\\
		0&0
	\end{bmatrix}$ with $T_1\in\mathcal{B}(\mathcal{K})$, $T_2\in\mathcal{B}(\mathcal{L},\mathcal{K})$ be as in Lemma 3.2 (iii) and $\sigma=\begin{bmatrix}
	0&0\\
	\sigma_3&\sigma_4
	\end{bmatrix}$ with $\sigma_3\in\mathcal{B}(\mathcal{K},\mathcal{L})$, $\sigma_4\in\mathcal{B}(\mathcal{L})$.
	
	\begin{enumerate}[\rm (i)]
		\item 
		If $\mathcal{R}(T_1)\subset\mathcal{R}(T_2)$ then we have the following.
		
		\begin{enumerate}[\rm (a)]
			\item 
			$T$ can be expressed as the product $E_1E_2$ of two idempotents $E_1=\begin{bmatrix}
				I_\mathcal{K}&T_2\\
				0&0
			\end{bmatrix}$, $E_2=\begin{bmatrix}
				0&0\\
				C&I_\mathcal{L}
			\end{bmatrix}$, where $C$ is given by Theorem D to satisfy $T_1=T_2C$.
			
			\item 
			${\rm dim}\mathcal{K}_1\le {\rm dim}\mathcal{L}$.
		\end{enumerate}
		
		\item 
		If $\mathcal{R}(\sigma_4)\subset\mathcal{R}(\sigma_3)$, then we have the following.
		\begin{enumerate}[\rm (a)]
			\item 
			$\sigma$ can be expressed as the product $F_1F_2$ of the idempotents $F_1$ and $F_2$ in $\mathcal{A}$ given by
			$F_1=\begin{bmatrix}
				0&0\\
				\sigma_3&I_\mathcal{L}
			\end{bmatrix}$, $F_2=\begin{bmatrix}
			I_\mathcal{K}&B\\
			0&0
			\end{bmatrix}$, where $B\in\mathcal{B}(\mathcal{L},\mathcal{K})$ is given by Theorem D to satisfy $\sigma_y=\sigma_3B$.
			
			\item 
			  If $\mathcal{L}_1$ is the closure of $\mathcal{R}(\sigma)$, then ${\rm dim}\mathcal{L}_1\le {\rm dim}\mathcal{K}$.
		\end{enumerate}
	\end{enumerate}
\end{theorem}

\begin{proof}
	(i) (a) and (ii) (a) come from direct application of Theorem D and the above lemma \ref{tm:2}.
	
	We only have to prove (ii) (b) and note that (i) (b) is similar. Indeed, $\mathcal{R}(\sigma_3)=\mathcal{R}(\sigma)$ is dense in $\mathcal{L}_1$. But $\mathcal{R}(\sigma_3)$ has dimension less than or equal to that of $\mathcal{K}$. So, ${\rm dim}\mathcal{L}_1\le {\rm dim}\mathcal{K}$.
\end{proof}

%

Example 2.6 (iv) (c) may be adapted to lead us to the following Theorem.

\begin{theorem}\label{tm:3}
	Let $T=\begin{bmatrix}
		T_1&T_2\\
		0&0
	\end{bmatrix}$ in $X=\mathcal{K}\oplus\mathcal{L}$ in some local representation. Suppose 
	\begin{enumerate}[\rm (a)]
		\item 
		$\mathcal{R}(T)=\mathcal{R}(T_1)$ i.e., $\mathcal{R}(T_2)\subset \mathcal{R}(T_1)$ and 
		
		\item 
		$\mathcal{K}\subset_e X$ so that there exists an injective $J \in  \mathcal{B}(\mathcal{K}, \mathcal{L})$ with a closed range $\mathcal{L}_1 \in \hat{\mathcal{K}}$ and a corresponding $J^*$ in $\mathcal{B}(\mathcal{L},\mathcal{K})$ with $J^*J=I_\mathcal{K}$. 
	\end{enumerate}
	Then $T$ is product of two idempotents in $\mathcal{A}$. Moreover, these factorizations are distinct for distinct $J$'s. Indeed, absolute value 1  scalar multiples of $J$ give infinitely many new choices to replace $(J, J^*)$.
\end{theorem}

\begin{proof}
	
	
	Because $\mathcal{R}(T_2)\subset \mathcal{R}(T_1)$, by Theorem D, there exists $B\in\mathcal{B}(\mathcal{L},\mathcal{K})$ that satisfies $T_2=T_1B$. Let
	$$U=\begin{bmatrix}
		I_\mathcal{K} & (T_1-I_\mathcal{K})J^*+B\\
		0 & 0
	\end{bmatrix}$$ and let $$V=\begin{bmatrix}
		I_\mathcal{K}-BJ & B-BJB\\
		J & JB
	\end{bmatrix}.$$
%
Then $T$ is the product of idempotents $U$ and $V$.
\end{proof}

\begin{remark}\label{rem:7}
Taking $T_2=0$ in Theorem 4.3, and $\mathcal{R}(B)\subset \mathcal{N}(T_1)$, for instance, $B=0$ in its proof, we arrive at ways to express $\begin{bmatrix}
	T_1 & 0\\
	0 & 0
\end{bmatrix}$ as a product of two idempotents.

\end{remark}

\subsection{Products of two idempotents via Consistency condions.}

Let us first make a few observations some of which look familiar or related to contents of Section 3, particularly Remark 3.5 (vi), Theorem 3.12 and the Remark 3.13.

\begin{theorem}\label{tm:MC}
	Let $0\ne T\in \mathcal{A}$ with a non-zero left annihilator. 
	\begin{enumerate}
		\item[($\alpha$)] 
		Let $T=\begin{bmatrix}
			T_1&T_2\\
			0&0
		\end{bmatrix}$ in $X=\mathcal{K}\oplus \mathcal{L}$ in a local block representation form as in Definition 3.10, which is not an idempotent with range $\mathcal{K}$ (i.e., $T_1\ne I_\mathcal{K})$. Then, $T$ is product of two idempotents if there exist $B\in \mathcal{B}(\mathcal{L},\mathcal{K})$, 	$C\in \mathcal{B}(\mathcal{K},\mathcal{L})$, $D\in \mathcal{B}(\mathcal{L})$, that satisfy any of the following equivalent sets (i), (ii), (iii), (iv) of conditions.
		\begin{enumerate}[\rm (i)]
			\item
			$\mathcal{C}_1$: $S=T_{B,C,D}=\begin{bmatrix}
				T_1-BC & T_2-BD\\
				C & D
			\end{bmatrix}$ is an idempotent.
			\item
			$\mathcal{C}_2$: $\mathcal{C}_1$ together with $TS=T$.
			\item 
			$\mathcal{C}_3$: Let (\ref{eqn1}) to (\ref{eqn6}) be as below.  

			(a) The set $\{(\ref{eqn1}), (\ref{eqn2}), (\ref{eqn3}), (\ref{eqn4})\}$ holds, or
			
			(b) The set $\{(4.1), (4.2), (4.3), (4.4),  (4.5), (4.6)\}$ holds.
			\begin{equation}\label{eqn1}
				(T_1-BC)^2+(T_2-BD)C=T_1-BC
			\end{equation}
			\begin{equation}\label{eqn2}
				(T_1-BC)(T_2-BD)+(T_2-BD)D=T_2-BD
			\end{equation}
			\begin{equation}\label{eqn3}
				C(T_1-BC)+DC=C
			\end{equation}
			\begin{equation}\label{eqn4}
				C(T_2-BD)+D^2=D
			\end{equation}
			\begin{equation}\label{eqn5}
				T_1(T_1-BC)+T_2C = T_1
			\end{equation} \mbox{ and }
			\begin{equation}\label{eqn6}
				T_1(T_2-BD)+T_2D = T_2
			\end{equation}
			\item
			$\mathcal{C}_4$: Let $U=T_1-BC$, $V=T_2-BD$ and (\ref{eqn1prime}) to (\ref{eqn6prime}) be as below. (a) the set of conditions \{(\ref{eqn1prime}), (\ref{eqn2prime}), (\ref{eqn3prime}), (\ref{eqn4prime})\} holds or
			
			(b) The set of conditions $\{(4.1'), (4.2'), (4.3'), (4.4'), (4.6'), (4.6')\}$ holds
			
			\begin{equation}\tag{\ref{eqn1}$'$}\label{eqn1prime}
				VC=(I_\mathcal{K}-U)U \equiv U(I_\mathcal{K}-U)
			\end{equation}
			\begin{equation}\tag{\ref{eqn2}$'$}\label{eqn2prime}
				UV = V (I_\mathcal{L}-D)
			\end{equation}
			\begin{equation}\tag{\ref{eqn3}$'$}\label{eqn3prime}
				CU=(I_\mathcal{L}-D)C
			\end{equation}
			\begin{equation}\tag{\ref{eqn4}$'$}\label{eqn4prime}
				CV= D(I_\mathcal{L}-D)\equiv (I_\mathcal{L}-D)D
			\end{equation}
			\begin{equation}\tag{\ref{eqn5}$'$}\label{eqn5prime}
				T_2C=T_1(I_\mathcal{K}-U)
			\end{equation}
			\begin{equation}\tag{\ref{eqn6}$'$}\label{eqn6prime}
				T_1 V=T_2(I_\mathcal{L}-D)
			\end{equation}
			
	\end{enumerate}
	\item[($\beta$)] 
	For the only if part, there exists some decomposition $X=\mathcal{K}\oplus\mathcal{L}$, a topological sum, with $\mathcal{K}\supset \mathcal{K}_0$ and conditions stated in ($\alpha$) above.
	\end{enumerate}
\end{theorem}
\begin{proof}
	\begin{enumerate}
		\item[($\alpha$)]
		(i) and (ii) follow from Remark 3.13.\\		
		(iii)
			[\rm (a)]
			For the set \{(\ref{eqn1}), (\ref{eqn2}), (\ref{eqn3}), (\ref{eqn4})\}, it follows from (i) on expanding $S^2=S$. (\ref{eqn5}) and (\ref{eqn6}) follow when we expand $TS=T$.
		(iv) The equation (\ref{eqn1}) can be re-written as $U^2+VC=U$, i.e., $VC=U-U^2=(I_\mathcal{K}=U)U\equiv U(I_\mathcal{K}-U)$ which becomes (\ref{eqn1prime}). 
		
		Similar computations give the remaining part.
		
		\item[($\beta$)]
		 Suppose $T=QQ^\prime$ for some $Q, Q^\prime$ in $\mathcal{E}$. Let $\mathcal{K}=\mathcal{R}(Q)$ and $\mathcal{L}=\mathcal{N}(Q)$.
	\end{enumerate}
\end{proof}

\begin{corollary}
	 Following the notation of Theorem 4.5 above, we have more sets of equivalent conditions: Any one of the sets
	 
	$(\{\ref{eqn3}), (\ref{eqn4}), (\ref{eqn5}), (\ref{eqn6})\}$, $\{(\ref{eqn3}), (\ref{eqn4}), (\ref{eqn1}), (\ref{eqn6})\}$, $\{(\ref{eqn3}), (\ref{eqn4}), (\ref{eqn2}), (\ref{eqn5})\}$ holds. 
	\{(\ref{eqn3prime}), (\ref{eqn4prime}), (\ref{eqn5prime}), (\ref{eqn6prime})\}, \{(\ref{eqn3prime}), (\ref{eqn4prime}), (\ref{eqn1prime}), (\ref{eqn6prime})\}, \{(\ref{eqn3prime}), (\ref{eqn4prime}), (\ref{eqn2prime}), (\ref{eqn5prime})\} holds.
\end{corollary}
\begin{proof}
	\begin{enumerate}
	\item[(a)] We rewrite equations (4.1) to (4.6) by making needed shifting to make R.H.S zero in each and call the new equations by  the same numbering for notational convenience. We note that L.H.S of (\ref{eqn5})$-$L.H.S of (\ref{eqn1})$= B[C(T_1-BC)+DC-C]$, So, if (\ref{eqn3}) holds then (\ref{eqn1}) and (\ref{eqn5}) become equivalent.
	\item[(b)]
	Now, L.H.S of (\ref{eqn6})$-$(\ref{eqn2})$=B[C(T_2-BD)+D^2-D].$ So if (\ref{eqn4}) holds then (\ref{eqn2}) and (\ref{eqn6}) become equivalent.
	\item[(c)]
	Remaining part of (iii) now follows from (a), (b) and (c) above. To elaborate, the system $\{(4.1), (4.2), (4.3), (4.4)\}$ is equivalent to the system where (4.1) is replaced by (4.5) or/and (4.2) is replaced by (4.6). This gives that the  remaining three systems in (iii) are as good as the first one.
	\end{enumerate}
\end{proof}

\begin{remark}
	We now proceed to identify $T$'s which satisfy the conditions set out in Theorem \ref{tm:MC} and figure out their factorizations into two idempotents. We work out details for some specific cases. For this purpose, we have already discarded  $(C,D)=(0,0)$ in Remark 3.13 (ii).  Also we have seen a situation in Theorem \ref{cor:1} above, viz., $\mathcal{R}(T_1)\subset \mathcal{R}(T_2)$ for which $B=T_2$, $D=I_\mathcal{L}$, and $C$ is such that $T_1=T_2C$ does work. We give more factoriztions in the results that follow.
	
%
\end{remark}


The proof of the following theorem follows from Remark 3.13 and direct computations.

\begin{theorem}\label{tm:9}
	Let $T=\begin{bmatrix}
		T_1 & T_2\\
		0 & 0
	\end{bmatrix}$  in $X=\mathcal{K}\oplus \mathcal{L}$ in some local block representation with $T_1\ne I_\mathcal{K}$ and $\mathcal{N}(T_1)\ne \{0\}$.
	 
	\begin{enumerate}[\rm (i)]
		\item Suppose $\mathcal{R}(T_1)\subset \mathcal{R}(T_2)$ and $C\in\mathcal{B}(\mathcal{L},\mathcal{K})$ as given by Theorem D to satisfy $T_1=T_2C$ and $\mathcal{N}(C)=\mathcal{N}(T_1)$. Let $V\in\mathcal{B}(\mathcal{L},\mathcal{K})$ with $\mathcal{R}(V)\subset \mathcal{N}(T_1)$. Then $T$ can be expressed as the product $E_1E_2$ of two idempotents $E_1=\begin{bmatrix}
			I_\mathcal{K} & T_2-V\\
			0 & 0
		\end{bmatrix}$, $E_2=\begin{bmatrix}
			VC & V\\
			C & I_\mathcal{L}
		\end{bmatrix}$. Furthermore, there are uncountably many $V$'s and therefore, uncountably many factorizations corresponding to these $V$'s.
		\item If $T_1$ is an idempotent, then $T$ can be expressed as a product of two idempotents in uncountable number of ways: $T=E_1E_2$, $E_1=\begin{bmatrix}
				I_\mathcal{K} & T_2-V\\
				0 & 0
			\end{bmatrix}$, and $E_2=\begin{bmatrix}
				T_1+VC & V\\
				C & I_\mathcal{L}
			\end{bmatrix}$, where $V\in\mathcal{B}(\mathcal{L}, \mathcal{K})$ and $C\in \mathcal{B}(\mathcal{K},\mathcal{L})$ satisfying $\mathcal{R}(V)\subset \mathcal{N}(T_1),(\mathcal{R}(V)\cup\mathcal{R}(T_1))\subset\mathcal{N}(C)$ and $\mathcal{R}(C)\subset \mathcal{N}(T_2)$. 
		\end{enumerate}
	\end{theorem}
\begin{remark}
	We start with the case $D=I_\mathcal{L}$ and look for more solutions envisaged in Theorem \ref{tm:MC} (iv), which becomes
	
	\begin{enumerate}[\rm (i)]
		\item 
		$V=T_2-BD=T_2-B$. So, $B=T_2-V$.
		
		\item 
		\begin{align*}
			U & = T_1-BC\\
			& = T_1-(T_2-V)C\\
			& = (T_1-T_2C)+VC\\
			& = W+VC,
		\end{align*}
		where $W=T_1-T_2C$.
		
		\item 
		The system (\ref{eqn1prime}) to (\ref{eqn6prime}) becomes 
		
		\begin{align}
			& VC = U(I_\mathcal{K}-U)\tag{\ref{eqn1}$^{\prime\prime}$}\label{eqn1Dprime}\\
			& WV+VCV=0\tag{\ref{eqn2}$^{\prime\prime}$}\label{eqn2Dprime}\\
			& CW+CVC=0\tag{\ref{eqn3}$^{\prime\prime}$}\label{eqn3Dprime}\\
			& CV=0\tag{\ref{eqn4}$^{\prime\prime}$}\label{eqn4Dprime}\\
			& T_1W+T_1VC=W\tag{\ref{eqn5}$^{\prime\prime}$}\label{eqn5Dprime}\\
			& T_1V=0.\tag{\ref{eqn6}$^{\prime\prime}$}\label{eqn6Dprime}
		\end{align}
		
		\item 
		Now (\ref{eqn3Dprime}) and (\ref{eqn4Dprime}) can be replaced by (\ref{eqn3Tprime}) and (\ref{eqn4Dprime}), with (\ref{eqn3Tprime}) as below.
		
		\begin{equation}\tag{\ref{eqn3}$^{\prime\prime\prime}$}\label{eqn3Tprime}
			CW=0.
		\end{equation}
		Also, (\ref{eqn5Dprime}) and (\ref{eqn6Dprime}) can be replaced by (\ref{eqn5Tprime}) and (\ref{eqn6Dprime}) with (\ref{eqn5Tprime}) as below.
		
		\begin{equation}\tag{\ref{eqn5}$^{\prime\prime\prime}$}\label{eqn5Tprime}
			T_1W=W,\ \text{i.e.,}\ (I_\mathcal{K}-T_1)W=0.
		\end{equation}
		
		\item 
		We may concentrate on (\ref{eqn3Tprime}), (\ref{eqn4Dprime}), (\ref{eqn5Tprime}), and (\ref{eqn6Dprime}).
		
		Let $V\in\mathcal{B}(\mathcal{L},\mathcal{K})$ with $\mathcal{R}(V)\subset \mathcal{N}(T_1)$ keeping (\ref{eqn6Dprime}) in mind. We note that $V=0$ if $\mathcal{N}(T_1)=\{0\}$. Next let $C\in\mathcal{B}(\mathcal{K},\mathcal{L})$ that satisfies $\mathcal{R}(V)\subset \mathcal{N}(C)$ keeping (\ref{eqn4Dprime}) in mind. The conditions (\ref{eqn3Tprime}) and (\ref{eqn5Tprime}) are  written as 
		\begin{align}
			& C(T_1-T_2C)=0\tag{C1}\label{eqnC1}\\
			& (I_\mathcal{K}-T_1)(T_1-T_2C)=0\tag{C2}\label{eqnC2}
		\end{align}
		and are called consistency conditions.
		
		Suppose (\ref{eqnC1}) and (\ref{eqnC2}) are satisfied. Then, using the fact that $CV=0$ and $T_1V=0$, we get that $T=E_1E_2$, a product of two idempotents $E_1=\begin{bmatrix}
			I_\mathcal{K} & T_2-V\\
			0&0
		\end{bmatrix}$, $E_2=\begin{bmatrix}
			T_1-T_2C+VC & V\\
			C & I_\mathcal{L}
		\end{bmatrix}$, which are distinct if $V$'s are so. For that $\mathcal{N}(T_1)\ne \{0\}$ is necessary.
		
		\item 
		We record suitable obvious subcases for which the consistency conditions (\ref{eqnC1}) and (\ref{eqnC2}) are satisfied.
		
		\begin{enumerate}[\rm (a)]
			\item 
			$T_1-T_2C=0$
			
			\item 
			$T_2C=0, CT_1=0,$ and $T_1$ is an idempotent.
		\end{enumerate}
		\item 
		Now (vi) (a) holds for $C$ $\Leftrightarrow$ $\mathcal{R}(T_1)\subset \mathcal{R}(T_2)$ by applying Theorem D.  In this case, there exists a unique $C$ satisfying $\mathcal{N}(T_1)=\mathcal{N}(C)$.  We fix such a $C$ now. So, all we have to do is to choose any $V\in\mathcal{B}(\mathcal{L},\mathcal{K})$ with $\mathcal{R}(V)\subset \mathcal{N}(T_1)$.
		
		We note that all this is possible only if ${\rm dim}\mathcal{K}_1\le {\rm dim}\mathcal{L}$. As already noted $V=0$ if $\mathcal{N}(T_1)=\{0\}$ and there are uncountably many $V$'s available if $\mathcal{N}(T_1)\ne \{0\}$.
		
		\item 
		Next (vi) (b) just requires extra case in (v) while choosing $C$, we have to make such that $\mathcal{N}(C)$ contains $\mathcal{R}(V)\cup\mathcal{R}(T_1)$ and $\mathcal{R}(C)\subset \mathcal{N}(T_2)$.
		
		Now we come to the case that $T_1$ is an idempotent. So, $\mathcal{N}(T_1)=\{0\}\Leftrightarrow T_1=I_\mathcal{K}$. But if $T_1=I_\mathcal{K}$ then $T$ is an idempotent and we may ignore this case and ensure $\mathcal{N}(T_1)\ne \{0\}$.
		
		We note that if $\mathcal{N}(T_2)=\{0\}$, then $C=0$.
	\end{enumerate}
\end{remark}

\begin{theorem}\label{tm:12}
	Let $T=\begin{bmatrix}
		T_1&T_2\\
		0 & 0
	\end{bmatrix}$ in $X=\mathcal{K}\oplus \mathcal{L}$ in some local block representation. Suppose that 
	\begin{enumerate}[\rm (a)]
		\item 
		$T_1$ is an idempotent $\ne I_\mathcal{K}$ and 
		
		\item 
		Either $\mathcal{N}(T_2)\ne \{0\}$ or $\mathcal{R}(T_1)\cap\mathcal{R}(T_2)\ne\{0\}$.
		
		\item 
		${\rm dim}\mathcal{L}\ge 2$.
	\end{enumerate}
	Then we have the following.
	\begin{enumerate}[\rm (i)]
		\item 
		There exist uncountably many idempotents $D$ with $\{0\}\ne\mathcal{N}(D)\ne\mathcal{L}$ and $\mathcal{N}(D)\subset\mathcal{N}(T_2-T_1T_2)$.
		
		\item 
		Let $D$ be as in (i) above. Then $T$ can be expressed as a product of two idempotents as follows.
		\begin{enumerate}[\rm (a)]
			\item 
			$T=E_1E_2$,
			
			\item 
			$T=E_1^\prime E_2$,
		\end{enumerate}
		where 
		$$E_1=\begin{bmatrix}
			I_\mathcal{K} & T_2\\
			0 & 0
		\end{bmatrix},\ E_2=\begin{bmatrix}
		T_1 & T_2(I_\mathcal{L}-D)\\
		0 & 0
		\end{bmatrix}\ \mbox{and}\ E_1^\prime=\begin{bmatrix}
		I_\mathcal{K} & T_2D\\
		0 & 0
		\end{bmatrix}.$$
		
		\item 
		$E_1=E_1^\prime\Leftrightarrow\mathcal{N}(D)\subset\mathcal{N}(T_2)\implies\mathcal{N}(T_2)\ne\{0\}$.
		
		\item 
		For distinct $D$'s as in (i) above, the corresponding $E_2$'s are distinct.
		
		\item 
		$T$ can be expressed as a product of two idempotents in uncountably many ways.
	\end{enumerate}
\end{theorem}
\begin{proof}\label{eg:9}
	Once again direct computations  give  the results. 
\end{proof}

 In the next remark we give some more explanations related to the above results.
	
\begin{remark}	
	We now consider $T$ defined on $X=\mathcal{K}\oplus \mathcal{L}$, $T_1\ne I_\mathcal{K}$ so that $T$ is not an idempotent with range equal to $\mathcal{K}$ and consider the case $C=0$, $D\ne0$ other than those considered above.
	\begin{enumerate}[\rm (i)]
		\item 
		So $0\ne D\ne I_\mathcal{L}$.
		
		\item
		We work in the context of Theorem 4.5 (iv). So, $U=T_1$ and the system (\ref{eqn1prime})-(\ref{eqn6prime}) becomes 
		\begin{align}
			& T_1(I_\mathcal{K}-U)=0,\tag{I}\label{eqnI}\\
			& T_1V = V(I_\mathcal{L}-D),\tag{II}\label{eqnII}
		\end{align}
		(\ref{eqn3prime}) is satisfied,
		\begin{align}
			& (I_\mathcal{L}-D)D=0\tag{IV}\label{eqnIV}\\
			& T_1(I_\mathcal{K}-T_1)=0\tag{V$\equiv$I}\label{eqnV}\\
			& T_1V=T_2(I_\mathcal{L}-D).\tag{VI}\label{eqnVI}
		\end{align}
		
		\item 
		(\ref{eqnI}) holds $\Leftrightarrow$ $T_1$ is an idempotent. In view of $T_1\ne I_\mathcal{K}$, either $T_1$ is zero or ${\rm dim}(\mathcal{K})\ge 2$. Also, (\ref{eqnIV}) holds $\Leftrightarrow$ $D$ is an idempotent. In this case, (i) forces ${\rm dim}(\mathcal{L})\ge2$. 
		
		(\ref{eqnVI}) can be rewritten as 
		\begin{equation}\tag{VI$^\prime$}\label{eqnVIprime}
			T_1T_2-T_2=(T_1B-T_2)D.
		\end{equation}
		In other words, the system is $T_1$ and $D$ are idempotents, $T_1\ne I_\mathcal{K}$, $0\ne D\ne I_\mathcal{L}$ and (\ref{eqnVIprime}) holds. We may call (\ref{eqnVIprime}) as consistency condition C3.
		
		\item 
		Let $T,B,D$ be as in (iii) above satisfying C3. Then $T$ is the product $E_1E_2$ of two idempotents
		$$E_1=\begin{bmatrix}
			I_\mathcal{K} & B\\
			0 & 0
		\end{bmatrix},\quad E_2=\begin{bmatrix}
			T_1 & T_2-BD\\
			0 & D
		\end{bmatrix}.$$
		
		\item 
		For $B=T_2$ or $T_2D$, we have $BD=T_2D$. So, RHS of (\ref{eqnVIprime}) = $T_1T_2D-T_2D=(T_1T_2-T_2)D$. So, (\ref{eqnVIprime}) is satisfied if and only if
		\begin{align}
			& (I_\mathcal{K}-T_1)T_2(I_\mathcal{L}-D)=0\nonumber\\
			\Leftrightarrow & \mathcal{N}(D)=\mathcal{R}(I_\mathcal{L}-D)\subset \mathcal{N}((I_\mathcal{K}-T_1)T_2)\tag{VI$^{\prime\prime\prime}$}\label{eqnVITprime}.
		\end{align}	
		Because $\{0\}\ne \mathcal{N}(D)\ne \mathcal{L}$, this requires 
		\begin{align*}
			& \mathcal{N}((I_\mathcal{K}-T_1)T_2)\ne \{0\}\\
			\Leftrightarrow & \exists x\ne 0\ \mbox{with}\ T_2x=t_1T_2x\\
			\Leftrightarrow & \exists x\ne 0\ \mbox{with}\ T_2x\in\mathcal{R}(T_1)\\
			\Leftrightarrow & \mathcal{N}(T_2)\ne \{0\}\ \mbox{or}\ \mathcal{R}(T_2)\cap\mathcal{R}(T_1)\ne \{0\}.
		\end{align*}
		Also, if this holds we may consider an idempotent $D$ that satisfies (\ref{eqnVITprime}). There are uncountably many such $D$'s. The corresponding idempotents are 
		$$E_1=\begin{bmatrix}
			I_\mathcal{K} & T_2\\
			0 & 0
		\end{bmatrix},\quad E_2=\begin{bmatrix}
			T_1 & T_2-T_2D\\
			0 & D
		\end{bmatrix}$$
		and 
		$$E_1^\prime=\begin{bmatrix}
			I_\mathcal{K} & T_2D\\
			0 & 0
		\end{bmatrix}, \quad E_2^\prime=\begin{bmatrix}
			T_1 & T_2-T_2D\\
			0 & D
		\end{bmatrix}.$$
		So, $E_2=E_2^\prime$ but $E_1=E_1^\prime$ $\Leftrightarrow$ $T_2=T_2D$ i.e., $T_2(I_\mathcal{L}-D)=0$ i.e., $\mathcal{N}(D)=\mathcal{R}(I_\mathcal{L}-D)\subset \mathcal{N}(T_2)$.
		In any case, $E_2$'s are different for different $D$'s 
	\end{enumerate}
\end{remark}

\begin{example}\label{eg:16}
	We continue with our discussion in the context of Remarks 4.8.and 4.10. We now come to the case $C\ne 0$, $D\ne 0$, $D\ne I_\mathcal{L}$, $C$ and $D$ both invertible.
	\begin{enumerate}[\rm (i)]
		\item 
		This requires 
		\begin{enumerate}[\rm (a)]
			\item 
			${\rm dim}\mathcal{K}={\rm dim}\mathcal{L}$ and
			
			\item 
			if ${\rm dim}X=n$ then this forces $n$ even and ${\rm rank}(T)=\frac{n}{2}$.
		\end{enumerate}

		\item 
		We work in the context of  Theorem \ref{tm:MC} (iv). Then (\ref{eqn1prime}) to (\ref{eqn4prime}) become equivalent to
		\begin{align}
			V & = U(I_\mathcal{K}-U)C^{-1}\tag{$\alpha$}\label{eqnalpha}\\
			V & = (I_\mathcal{K}-U)VD^{-1}\tag{$\beta$}\label{eqnbeta}\\
			U & = C^{-1}(I_\mathcal{L}-D)C\tag{$\gamma$}\label{eqngamma}\\
			  & = I_\mathcal{K}-C^{-1}DC\nonumber\\
			V & = C^{-1}D(I_\mathcal{L}-D)\tag{$\delta$}\label{eqndelta}.
		\end{align}
		Let us assume (\ref{eqngamma}) and (\ref{eqndelta}).
		
		Then RHS of (\ref{eqnalpha}) is
		\begin{equation*}
			(I_\mathcal{K}-C^{-1}DC)C^{-1}DCC^{-1}=C^{-1}D-C^{-1}D^2=C^{-1}D(I_\mathcal{L}-D)=V=\mbox{LHS of (\ref{eqnalpha})}.
		\end{equation*}
		So, (\ref{eqnalpha}) holds.
		
		\begin{equation*}
			\mbox{RHS of (\ref{eqnbeta})}=(C^{-1}DC)C^{-1}D(I_\mathcal{L}-D)D^{-1}=C^{-1}D(I_\mathcal{L}-D)=V=\mbox{LHS of (\ref{eqnbeta})}.
		\end{equation*}
		So, (\ref{eqnbeta}) is satisfied.
		
		Thus the system can be replaced by (\ref{eqngamma}) and (\ref{eqndelta}).
		
		\item 
		Let us assume (\ref{eqngamma}) and (\ref{eqndelta}) hold. Then $T_1-BC=U=I_\mathcal{K}-C^{-1}DC$. So, $BC=T_1-I_\mathcal{K}+C^{-1}DC$. Thus $B=T_1C^{-1}-C^{-1}+C^{-1}D=T_1C^{-1}-C^{-1}(I_\mathcal{L}-D)$. Now, $T_2-BD=V=C^{-1}D(I_\mathcal{L}-D)$. So, $BD=T_2-C^{-1}(I_\mathcal{L}-D)D$. Thus, $B=T_2D^{-1}-C^{-1}(I_\mathcal{L}-D)$. Hence, we get a consistency condition $T_1C^{-1}-C^{-1}(I_\mathcal{L}-D)=T_2D^{-1}-C^{-1}(I_\mathcal{L}-D)$ i.e., $T_1C^{-1}=T_2D^{-1}$. 
		
		In particular, $\mathcal{R}(T_1)=\mathcal{R}(T_2)$. 
		
		Hence we may state the consistency condition as:
		\begin{quote}
			For some invertible operators $C^\prime\in\mathcal{B}(\mathcal{L},\mathcal{K})$ and $D^\prime\in\mathcal{B}(\mathcal{L})$,
			\begin{equation}\tag{C5}\label{eqnC5}
				 T_1C^\prime=T_2D^\prime.
			\end{equation}
		\end{quote}
		Then we may choose $C=C^{\prime^{-1}},\ D=D^{\prime^{-1}},\ B=T_1C^{-1}-C^{-1}(I_\mathcal{L}-D)$. Thus, we almost have the following theorem.
	\end{enumerate}
\end{example}

\begin{theorem}\label{tm:16}
	Let $T=\begin{bmatrix}
		T_1 & T_2\\
		0 & 0
	\end{bmatrix}$ in $X=\mathcal{K}\oplus \mathcal{L}$, and $T_1\ne I_\mathcal{K}$. Suppose $T_1C^{-1}=T_2D^{-1}$ for some invertible operators $C\in\mathcal{B}(\mathcal{K},\mathcal{L})$, $D\in\mathcal{B}(\mathcal{L})$. Then we have the following.
	\begin{enumerate}[\rm (i)]
		\item 
		${\rm dim}\mathcal{K}={\rm dim}\mathcal{L}$ and $\mathcal{R}(T_1)=\mathcal{R}(T_2)$.
		
		\item 
		$T$ is expressible as a product of two idempotents: $T=E_1E_2$,
		
		 $E_1=\begin{bmatrix}
			I_\mathcal{K} & T_1C^{-1}-C^{-1}(I_\mathcal{L}-D)\\
			0 & 0
		\end{bmatrix},\ E_2=\begin{bmatrix}
		I_\mathcal{K}-C^{-1}DC & C^{-1}(I_\mathcal{L}-D)D\\
		C & D
		\end{bmatrix}.$
		
		\item 
		If ${\rm dim}\mathcal{L}$ ($={\rm dim}\mathcal{K}$) $\ge 2$ then (ii) can be done in uncountably many ways.
	\end{enumerate}
\end{theorem}

\begin{proof}
	In view of Example \ref{eg:16} above, we 
	have only to prove (iii). Indeed there 
	exist uncountably many invertible
	$D^\prime\in\mathcal{B}(\mathcal{L})$. We 
	may put $D^{\prime\prime}=D^\prime D$, 
	$C^{\prime\prime}=D^\prime C$. Then, $T_1C^{\prime\prime^{-1}}=T_1C^{-1}D^{\prime^{-1}}=T_2D^{-1}D^{\prime^{-1}}=T_2D^{\prime\prime^{-1}}$. So, the idempotents corresponding to ($C^{\prime\prime},D^{\prime\prime}$) in place of $(C,D)$ do have their product as $T$ only.
\end{proof}

We conclude the discussion of operators expressible as product of two idempotents even though some more cases remain to be seen.

\section*{Acknowledgment}
Authors thank Mr. Umashankara Kelathaya for his unusual help in putting the manuscript in tex form.

\bibliographystyle{amsalpha}

%
%
%

\end{document}